\documentclass[twoside,12pt]{article}
\usepackage{color}

\usepackage{graphics}
\usepackage{graphicx}

\usepackage{amssymb,amsfonts}
\usepackage{amsmath}
\usepackage{amsthm}

\def\hpt{${\cal HPT}$ }
\def\hpp{${\cal HPP}$ }
\def\hps{${\cal HPS}$ }

\renewcommand{\vec}[1]{\mathbf{#1}}

\newcommand*\rfrac[2]{{}^{#1}\!/_{#2}}

\newtheorem{theorem}{Theorem}

\newtheorem{remark}{Remark}

\def\cim{Hyperbolic Pascal simplex}

\title{\bf \cim 
}
\author{ L\'aszl\'o N\'emeth\footnote{University of Sopron,  Institute of Mathematics, Hungary. \textit{nemeth.laszlo@uni-sopron.hu}}}
\date{}

\begin{document}

\maketitle \thispagestyle{empty}

\begin{abstract}
In this article we introduce a new geometric object called hyperbolic Pascal simplex. This new object is presented by the regular hypercube mosaic  in the 4-dimensional hyperbolic space. The definition of the hyperbolic Pascal simplex, whose hyperfaces are hyperbolic Pascal pyramids and faces are hyperbolic Pascals triangles, is a natural generalization of the definition of the hyperbolic Pascal triangle and pyramid. We describe the growing of the hyperbolic Pascal simplex considering the numbers and the values of the elements. Further figures illustrate the stepping from a level to the next one.   \\[1mm]
\end{abstract}


\section{Introduction}\label{sec:introduction} 

There are several generalizations of  Pascal's arithmetical triangle (see, for instance \cite{BNSz,BSz,SV}),  there is among them the family of hyperbolic Pascal triangles. This  new type is based on the hyperbolic regular mosaics denoted by Schl\"afli's symbol $\{p,q\}$, where $(p-2)(q-2)>4$ (see \cite{C}). Each regular mosaic induces a so-called hyperbolic Pascal triangle (\hpt\!), it is detailed only for regular squared mosaics $\{4,q\}$  in \cite{BNSz,NSz_alter,NSz2,NSz_Power,N_fibo}. Obviously, the classical Pascal's triangle is connected to the Euclidean square mosaic $\{4,4\}$.

The 3-dimensional analogue of  Pascal's original triangle is the well-known Pascal's pyramid (or more precisely Pascal's tetrahedron). 
Its levels are triangles and the numbers along the three edges of the $n^{\text{th}}$ level are the numbers of the $n^{\text{th}}$ lines of  Pascal's triangle. 
Each number inside in any levels is the sum of the three adjacent numbers on the level above \cite{ANVI,B,har}.
In \cite{NL_hyppyr} a 3-dimensional variation, the hyperbolic Pascal (cube) pyramid (\hpp\!\!) is presented, which is based on the hyperbolic regular cube mosaic $\{4,3,5\}$. This object can also be considered as the a hyperbolic variation of the well-known (Euclidean) Pascal's pyramid which is built on the Euclidean regular cube mosaic $\{4,3,4\}$. In a special space, in $\mathbf{H}^2\!\times\!\mathbf{R}$, there is an interesting generalization of \hpp (see \cite{N_pyrH2R}).

In the 4-dimensional space the natural generalization of the square and the cube is the 4-dimensional hypercube. Coxeter \cite{C} showed that the hypercube generates regular mosaics not only in the Euclidean, but also in the hyperbolic 4-dimensional spaces. They are the mosaics $\{4,3,3,4\}$ and $\{4,3,3,5\}$, respectively. The 4-dimensional Euclidean variation, Pascal's simplex can be based on the Euclidean hypercube mosaic $\{4,3,3,4\}$.

In this article we present a 4-dimensional hyperbolic Pascal simplex (\hps\!\!) built on the mosaic $\{4,3,3,5\}$. The method of the discussion is similar to the discussion of the hyperbolic Pascal pyramid's (\cite{NL_hyppyr}) case and we apply some of its results. (We keep the usual notation and write the hyperbolic Pascal pyramid without an ``apostrophe", similarly to the case of  Pascal's classical triangle and the hyperbolic Pascal triangle.) 
We strictly follow the definitions and denotations of the hyperbolic Pascal triangle and pyramid in \cite{BNSz} and \cite{NL_hyppyr} (see also \cite{N_pyrH2R}).
We mention that in the 4-dimensional Euclidean and hyperbolic spaces there are 2 and 4 other regular mosaics with bounded cells and vertex figures, but we do not examine them in this article. We suppose, that their examinations can be similar due to the generalization of Pascal triangle.

\section{Construction of the hyperbolic Pascal simplex}

The 4-dimensional hyperbolic hypercube mosaic $\{4,3,3,5\}$ consists of hypercubes with Schl\"afli's symbol $\{4,3,3\}$. 
The vertex figures of the mosaic are 600-cells, $\{3,3,5\}$. The hypercube is well-known, but the 600-cell is not so much, thus we give some details of it. 
The cells of a 600-cell are tetrahedra ($\{3,3\}$), the (2-dimensional) faces are triangles and the neighbouring vertices of a vertex of a 600-cell form an icosahedron ($\{3,5\}$), so a vertex lies on 12 edges, 30 faces and 20 cells. Each edge is on 5 faces and 5 cells and each face connects 2 cells. 
The numbers of the vertices, edges, (2-dimensional) faces and cells of a 600-cell are 120, 720, 1200 and 600, respectively. 

Considering an arbitrary vertex $V$ of the mosaic, the number of hypercubes around $V$ is 600, as many as the number of the cells of the 600-cell, and the number of the mosaic edges from $V$ (degree of $V$) is 120, as many as the number of the vertices of the 600-cell. There are 20 hypercubes around a mosaic edge as there are 20 faces of an icosahedron, so there are 20 tetrahedra around a vertex on the 600-cell. See some other details of the hyperbolic hypercube mosaic in \cite{C,N_hypercube}.

Let us consider an arbitrary vertex of the mosaic, say vertex $V_0$, as the base vertex of \hps\!\!. We sign it by 1 as well and let us sign all the vertices of the mosaic by the numbers of the shortest paths from the considered vertex to the base vertex along mosaic edges. 
The shortest paths imply a digraph directed from $V_0$.
We define a convex part ${\cal P}$ of the mosaic the following way. First we take a hypercube with vertex $V_0$ and we consider the vertex $V_1$ opposite of $V_0$ in the hypercube. It is the furthest vertex of the hypercube from $V_0$ and it has the largest sign among the vertices, namely $24$.
Second we take the new hypercubes of the mosaic containing vertex $V_1$ and their vertices which have the largest signs in each hypercubes. Now we take again the hypercubes around these vertices and their vertices with the largest signs, and so on. Continuing this algorithm limitless, the set of these hypercubes gives ${\cal P}$.  
Finally, the vertices (labelled above)  and the edges (directed above) of ${\cal P}$ form an infinite digraph similar to an infinite simplex with a finite base vertex $V_0$. We name it hyperbolic Pascal simplex (\hps\!\!).  Obviously, the signs of the vertices are the sums of the signs of the incoming edges.

Let level~0 be the vertex $V_0$. Level~$n$ consists of the vertices of \hps whose edge-distances from $V_0$ are $n$-edge (the distance of the shortest path along the edges of ${\cal P}$ is $n$).   The shapes of the levels are tetrahedra.  
It is clear, the 3-dimensional and 2-dimensional faces on the outer boundaries are hyperbolic Pascal pyramids and hyperbolic Pascal triangles based on mosaics $\{4,5,4\}$ and $\{4,5\}$, respectively. The faces and edges of the $n^{\text{th}}$ level are the $n^{\text{th}}$ levels of \hpp and $n^{\text{th}}$ lines of \hpt\!\!, respectively. 
Figures~\ref{fig:layer0to2}, \ref{fig:layer2to3} and \ref{fig:layer3to4} show the growing from a level to the next one in case of some lower levels.
As each hyperface of \hps is \hpp\!\!, there are six types of vertices: $1$, $A$, $B$ (from the 2-dimensional faces, which are hyperbolic Pascal triangles -- see \cite{BNSz}) and $C$, $D$, $E$ in the hyperfaces, corresponding to \cite{NL_hyppyr}. The properties of the growing come partly from the examination of \hpt and \hpp\!\!. 
We denote them, respectively, by grey, red, cyan, blue, green and yellow colours in the figures.  
The colours of the other different types of the vertices inside \hps are also different. (See the definitions later. The first vertex inside of the simplex appears in level 4. It is the biggest, the purple sphere -- vertex type $F$ -- in the centre of Figure~\ref{fig:layer3to4}.) The numbers without colouring refer to vertices in the lower level in every figure. 
The graphs, growing from a level to the new one, contain graph-cycles with six nodes. These graph-cycles figure the convex hulls of the parallel projections of the cubes from the mosaic, where the direction of the projection is not parallel to any edges of the cubes. Moreover, the rhombic-dodecahedra are the 3-dimensional shadows of the hypercubes.  

\begin{figure}[thb!]
	\centering
	\includegraphics[scale=0.56]{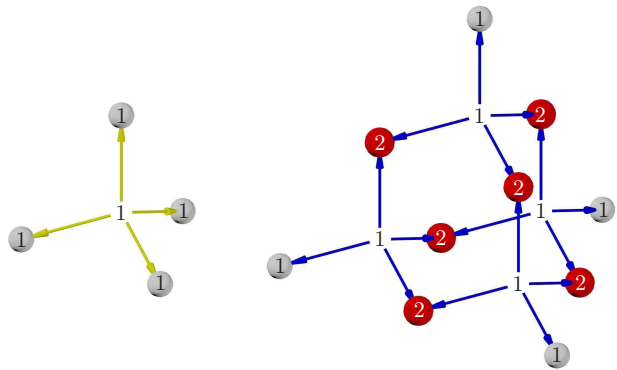}
	\caption{Connections between layers zero to one and one to two in \hps}
	\label{fig:layer0to2}
\end{figure}

\begin{figure}[htb!]
	\centering
	\includegraphics[scale=0.56]{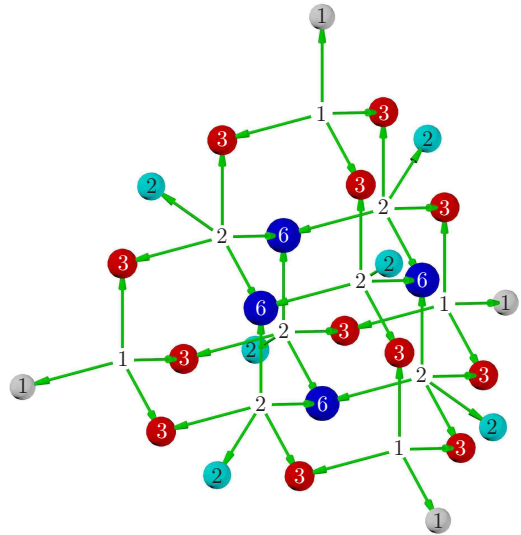}
	\caption{Connection between layers two and three in \hps}
	\label{fig:layer2to3}
\end{figure}

\begin{figure}[htb!]
	\centering
	\includegraphics[scale=0.75]{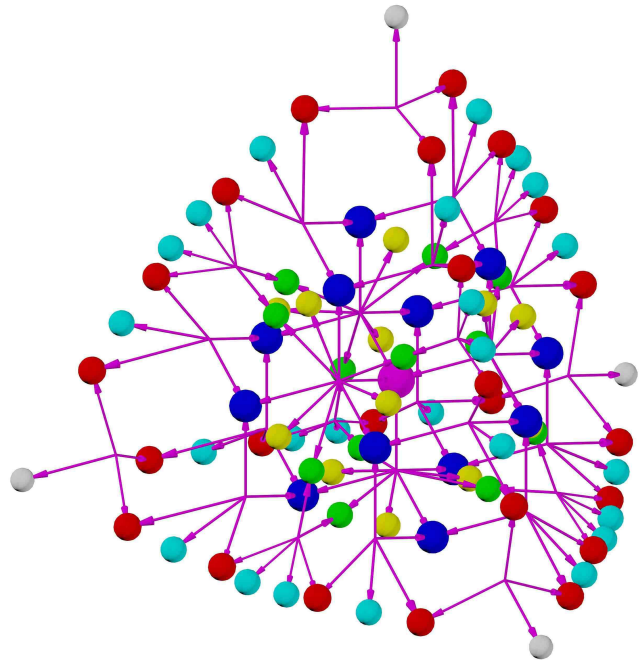}
	\caption{Connection between layers three and four in \hps}
	\label{fig:layer3to4}
\end{figure}

In the following we describe the method of growing of the hyperbolic Pascal simplex and we give the sum of the paths connecting vertex $V_0$ to level $n$.

\section{Growing of the hyperbolic Pascal simplex}

In the classical Pascal's simplex the number of the elements on level $n$ is $(n+1)(n+2)/2$ and its growing  from level $n$ to level $n+1$ is $n+2$. In this section we give the growing from level to level in the case of the hyperbolic Pascal simplex. 

The growing process on the outer 2- and 3-dimensional faces of \hps mostly comes from \cite{NL_hyppyr}. 
Figure~\ref{fig:gowing4d_edge} shows the growing of vertices types $1$, $A$ and $B$. For example, the centre figure illustrates that each vertex $A$ has two incoming edges, which could be types $1$, $A$ or $B$, and five outgoing edges, precisely two $A$, one $B$ and two $C$. The degrees of $1$, $A$ and $B$ are 5, 7 and 7, respectively.

\begin{figure}[ht!]
	\centering
	\includegraphics[scale=0.99]{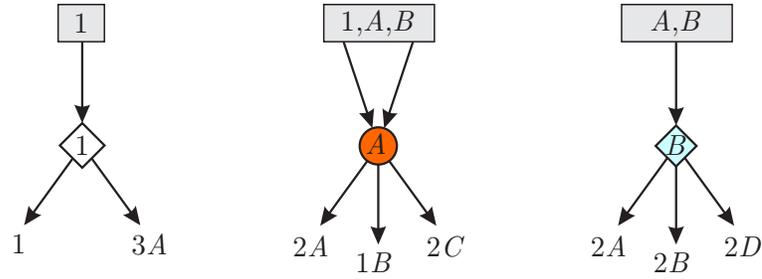}
	\caption{Growing of the 2-dimensional faces in \hps}
	\label{fig:gowing4d_edge}
\end{figure}

The vertices $C$, $D$ and $E$ are inside the 3-dimensional faces, which are \hpp\!\!-s and they have outgoing edges into the inside of \hps\!\!. (There are 12 hypercubes around these edges in the mosaic.) Let us denote these new types of vertices by $F$, $G$ and $H$ (see Figure~\ref{fig:gowing4d_face}).  So, the degrees of  vertices $C$, $D$ and $E$ are 13. (Compare Figures~\ref{fig:layer0to2}--\ref{fig:layer3to4} with Figures~\ref{fig:gowing4d_edge}--\ref{fig:gowing4d_face}.)

\begin{figure}[ht!]
	\centering
	\includegraphics[scale=0.99]{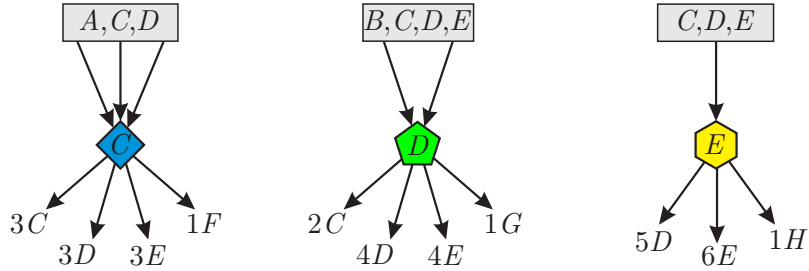}
	\caption{Growing of the 3-dimensional faces in  \hps}
	\label{fig:gowing4d_face}
\end{figure}

Now we examine the vertices inside \hps\!\!.   
All these vertices have four, three, two or one incoming edges from the previous level. We denote them by types $F$, $G$, $H$ and $K$, respectively. (In Figure~\ref{fig:layer3to4} vertex $F$ is coloured  purple.) Some vertices $F$, $G$ and $H$ have one incoming edge from vertices $C$, $D$ and $E$, respectively, and all vertices types $F$  connect only to inside vertices of \hps\!\!. 

In the following we give the number of the outgoing edges of these vertices using the classification of the vertices of the vertex figures with a 4-dimensional generalized method applied in \cite{NL_hyppyr}.
According to the previous section, the degrees of all these vertices are 120. 
Firstly, we consider a vertex type $K$ in level $i$ $(i\geq 5)$. In the mosaic, the neighbouring vertices to $K_i$  form a vertex figure, a 600-cell, whose all 120 vertices have mosaic edges joining to $K_i$. Among them there is only one vertex in level $i-1$, we denote it by $W_{i-1}$ (see Figure~\ref{fig:gowing4d_vertex_K}). The other vertices are in level $i+1$  and with the classification of them we can give the numbers of different types of vertices which imply the outgoing mosaic edges from $K_i$. (The mosaic edge $W_{i-1}K_i$ is an incoming edge to $K_i$.) If a vertex of a 600-cell has one common edge with  $W_{i-1}$, then it has two incoming mosaic edges. (We mention that the edges of the vertex figure are not the edges of the mosaic.) 
In Figure~\ref{fig:gowing4d_vertex_K} vertex $H_{i+1}$ connects to $W_{i-1}$, so there is a $W_i$ mosaic vertex (not in the vertex figure) which has common mosaic edges not only with $W_{i-1}$ but also with $H_{i+1}$. This way  $H_{i+1}$ has two incoming edges from $W_i$ and $K_i$ from level $i$, thus its type is $H$. All the vertices of the 600-cell which are adjacent to $H_{i+1}$ form an icosahedron and the type of all its vertices are $H$ (orange regular 9-gons in the figures). The type of the other 107 vertices of the 600-cell is $K$, as their incoming edges come from the considered vertex $K_i$.     

\begin{figure}[htb!]
	\centering
	\includegraphics[scale=0.99]{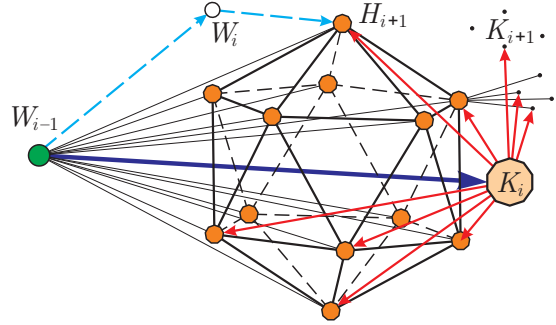}
	\caption{Growing method around vertex type $K$}
	\label{fig:gowing4d_vertex_K}
\end{figure}

Secondly, we take a vertex type $H$ in level $i$ $(i\geq 5)$, let it be $H_i$ in Figure~\ref{fig:gowing4d_vertex_H}.  $H_i$ has two incoming edges from level $i-1$, so there are two vertices on the same edge of the vertex figure in level $i-1$. We denote them by $W_{i-1}$ again. (Generally, $W$ denotes a vertex, whose type is not known or not important to know, while the index shows the level of the vertex.) Now we have to classify again the vertices of the vertex figure around the considered $H_i$ according to its neighbouring  vertices $W_{i-1}$. 
As there are 5 vertices in a 600-cell which are connected to both vertices by an edge, there are 5 vertices type $G$, they have 3 incoming edges from level $i$. (Recall, the mosaic edges and the icosahedron edges are different.) These 5 vertices are the intersections of the icosahedra connecting to the two $W_{i-1}$-s. 
In Figure~\ref{fig:gowing4d_vertex_H} we drew the icosahedron connecting to the left vertex $W_{i-1}$ and the vertices $G_{i+1}$ are denoted by blue regular 8-gones. The other $2\cdot 6$ vertices of the icosahedra are one-edge-long far from one of the $W_{i-1}$-s, so their types are $H$. The rest 101 vertices of the 600-cell are types $K$. On the right-hand side of the figure we highlighted the vertices $W_{i-1}$ and  $G_{i+1}$ and their subgraph structure in the 600-cell.

\begin{figure}[htb!]
	\centering
	\includegraphics[scale=0.99]{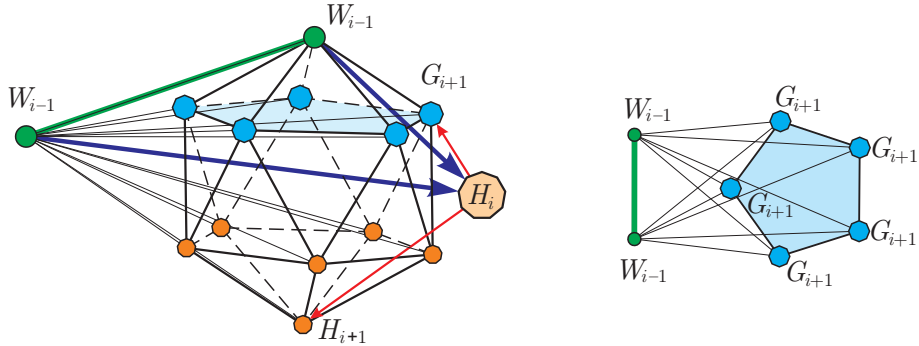}
	\caption{Growing method around vertex type $H$}
	\label{fig:gowing4d_vertex_H}
\end{figure}

Thirdly, we classify the vertices of the 600-cell in the case of a vertex $G_i$ $(i\geq 5)$. We mark a face of the 600-cell, its vertices are in level $i-1$ and from them start the mosaic's incoming edges into $G_i$. There are only two vertices type $F$, which are neighbouring to all the three $W_{i-1}$-s. Thus both $F_{i+1}$ have 4 incoming mosaic edges and Figure~\ref{fig:gowing4d_vertex_G} shows them with yellow regular 7-gones. For all the three edges of triangle $W_{i-1}W_{i-1}W_{i-1}$ connect  pentagons, so that the vertices are one-edge-long far from the endpoints of the edges. The type of their  two common vertices are $F$, the other $3\cdot 2$ vertices are $G$ (see on the right-hand side in Figure~\ref{fig:gowing4d_vertex_G}). The other $3\cdot 4$ vertices of the three icosahedra connecting to all $W_{i-1}$ are $H$. The number of the rest vertices of the vertex figure is 97, their types are $K$.

\begin{figure}[htb!]
	\centering
	\includegraphics[scale=0.99]{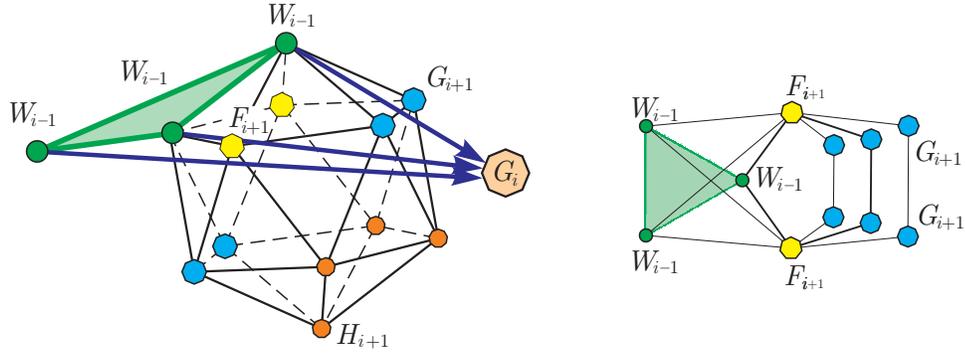}
	\caption{Growing method around vertex type $G$}
	\label{fig:gowing4d_vertex_G}
\end{figure}

Fourthly, we classify the vertices of the 600-cell in the case when the vertex is  $F_i$ $(i\geq 4)$. Now a tetrahedron is in level $i-1$ and from its vertices $W_{i-1}$ start the incoming mosaic edges into the considered $F_i$. For all four faces of tetrahedron $W_{i-1}W_{i-1}W_{i-1}W_{i-1}$ connect a vertex $F_{i+1}$, so that they are joining to three $W_{i-1}$, so they have 4 incoming mosaic edges from level $i$ (see Figure~\ref{fig:gowing4d_vertex_F}). That way along all the six edges of the tetrahedron there are 4 neighbouring vertices to the end of the edge. The rest, the fifth vertices are type $G$, the number of them is $6\cdot 1$. (There are 6 edges of a tetrahedron.) The other $4\cdot 3$ vertices of the four icosahedra connecting to vertices $W_{i-1}$ are type $H$. The rest vertices of the 600-cell are type $K$, they are 94 altogether.  

\begin{figure}[htb!]
	\centering
	\includegraphics[scale=0.99]{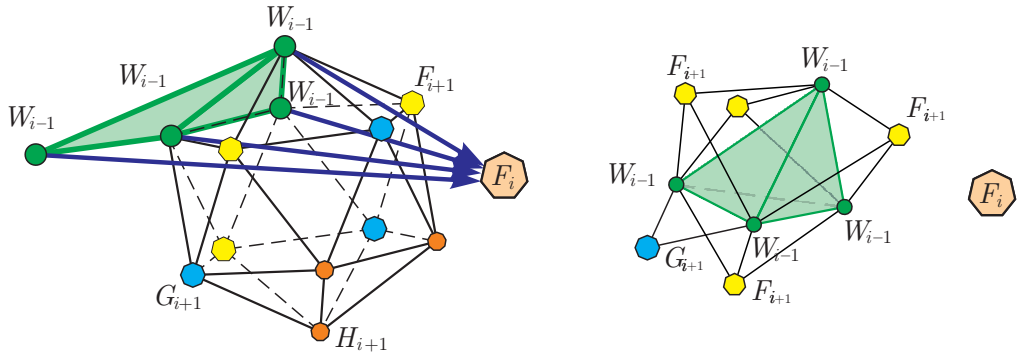}
	\caption{Growing method around vertex type $F$}
	\label{fig:gowing4d_vertex_F}
\end{figure}

Finally, in Figure~\ref{fig:gowing4d_simlex} the growing method is presented in the case of the inner vertices of \hps as a summing-up of discussions and figures. For example, the first graph shows that each vertex $F$ has four incoming edges, which could be types $C$, $F$ or $G$, and $120-4$ outgoing edges with 4 pieces $F$, 6 pieces of $G$, 12 pieces of $H$ and 94 pieces of $K$.  
\begin{figure}[h!]
	\centering
	\includegraphics[scale=0.99]{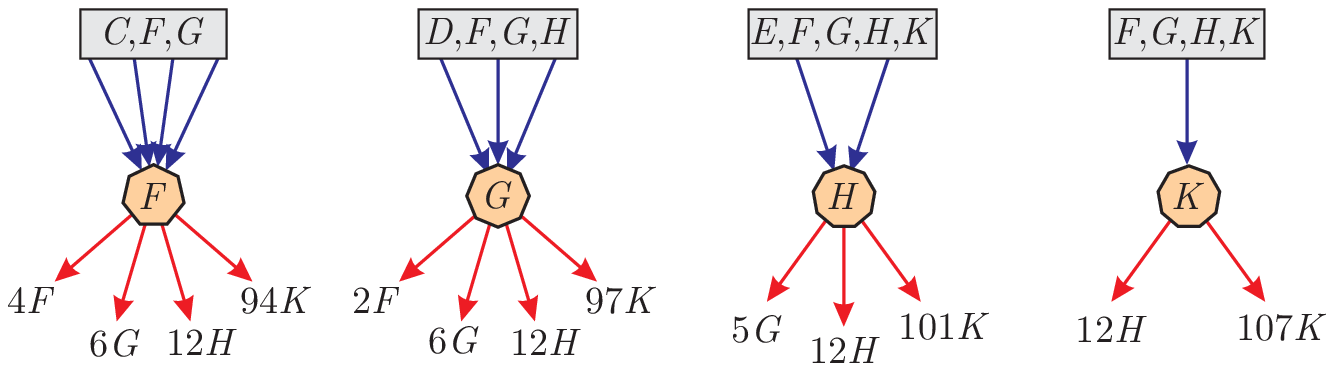}
	\caption{Growing inside \hps}
	\label{fig:gowing4d_simlex}
\end{figure}

Figure \ref{fig:gowing4d_edge}, \ref{fig:gowing4d_face} and \ref{fig:gowing4d_simlex} describe the growing method of \hps\!\!, but for summarising we have to consider the vertices in level $i+1$ without multiplicity. For example the new vertices $F$ are counted for all the vertices $C$, $F$ and $G$. So, for the correct calculation we correspond only one third of them to the examined vertices $C$, $F$ and $G$. Similarly for the cases of all vertices we eliminate the multiplicity.

Let us denote the sums of vertices types $A$, $B$, $C$,  $D$,  $E$, $F$, $G$, $H$, $K$ and $1$ on level $i$  by   $a_i$, $b_i$, $c_i$, $d_i$, $e_i$, $f_i$, $g_i$, $h_i$, $k_i$ and $v_i$, respectively.

Summarising the details $(i\geq4)$ and calculating the numbers of vertices in some lower levels $(i<4)$ in Table~\ref{table:typeof_vertices4d}, we proved the Theorem~\ref{th:4d_growing_type}.

\begin{theorem}\label{th:4d_growing_type}
	The growing of the numbers of the different types of  vertices in \hps is described by the system of linear homogeneous recurrence sequences  $(n\geq1)$
	\begin{equation}\label{eq:4d_seq01v}
	\begin{aligned}
	a_{n+1}&= \frac12 \left(2a_n+2b_n+3v_n\right),\\
	b_{n+1}&=  a_n+2b_n,\\
	c_{n+1}&= \frac13 \left(2a_n+3c_n+2d_n\right),\\
	d_{n+1}&= \frac12 \left(2b_n+3c_n+4d_n+5e_n\right),\\
	e_{n+1}&= 3c_n+4d_n+6e_n,\\
	f_{n+1}&= \frac14 \left(c_n+4f_n+2g_n\right),\\
	g_{n+1}&= \frac13 \left(d_n+6f_n+6g_n+5h_n\right),\\
	h_{n+1}&= \frac12 \left(e_n+12f_n+12g_n+12h_n+12k_n\right),\\
	k_{n+1}&=  94f_n+97g_n+101h_n +107k_n \\
	v_{n+1}&=  v_n,
	\end{aligned}
	\end{equation}
	with zero and $v_1=4$ initial values.
\end{theorem}

Moreover, let $s_n$ be the number of all the vertices in level $n$, so that $s_0=1$ and
\begin{equation}\label{eq:sn_4d}
s_{n}= a_n+b_n+c_n+d_n+ e_n+f_n+g_n+h_n+k_n+4, \qquad (n\geq1).
\end{equation}

Table \ref{table:typeof_vertices4d} shows the numbers of the different type of vertices on levels up to 10.  

\begin{table}[!htb]
	\centering \setlength{\tabcolsep}{0.3em}
	\begin{tabular}{|c||c|c|c|c|c|c|c|c|c|c|c|}
		\hline
		$n$   & 0 & 1 & 2  & 3  & 4  & 5   & 6     & 7       & 8         & 9            & 10            \\ \hline \hline
		$a_n$ & 0 & 0 & 6  & 12 & 24 & 54  & 132   & 336     & 870       & 2268         & 5928          \\ \hline
		$b_n$ & 0 & 0 & 0  & 6  & 24 & 72  & 198   & 528     & 1392      & 3654         & 9576          \\ \hline
		$c_n$ & 0 & 0 & 0  & 4  & 12 & 36  & 136   & 696     & 4512      & 33004        & 253260        \\ \hline
		$d_n$ & 0 & 0 & 0  & 0  & 12 & 96  & 708   & 5388    & 41868     & 328116       & 2579232       \\ \hline
		$e_n$ & 0 & 0 & 0  & 0  & 12 & 156 & 1428  & 11808   & 94488     & 747936       & 5899092       \\ \hline
		$f_n$ & 0 & 0 & 0  & 0  & 1  & 4   & 16    & 86      & 1111      & 70970        & 7610192       \\ \hline
		$g_n$ & 0 & 0 & 0  & 0  & 0  & 6   & 72    & 1702    & 137462    & 15061942     & 1694955086    \\ \hline
		$h_n$ & 0 & 0 & 0  & 0  & 0  & 12  & 774   & 79254   & 8862504   & 998747934    & 112617248352   \\ \hline
		$k_n$ & 0 & 0 & 0  & 0  & 0  & 94  & 12228 & 1395058 & 157449038 & 17755598218  & 2002190230214 \\ \hline
		$v_n$ & 1 & 4 & 4  & 4  & 4  & 4  & 4 & 4 & 4 & 4  & 4 \\ \hline
		$s_n$ & 1 & 4 & 10 & 26 & 89 & 534 & 15696 & 1494860 & 166593249 & 18770594046  & 2116518790936 \\ \hline
	\end{tabular}
	\caption{Number of types of the vertices $n\leq10$}
	\label{table:typeof_vertices4d}
\end{table}

\begin{theorem}\label{th:4d_numvertex4q}  
	The sequences $\{a_n\}$, $\ldots$, $\{k_n\}$  and $\{s_n\}$ can be described by the same ninth order linear homogeneous recurrence sequence 
	\begin{multline}\label{eq:4d_recur01}
	x_n=128x_{n-1}-1795x_{n-2}+8837x_{n-3}-19239x_{n-4}+19239x_{n-5}\\
	-8837x_{n-6}+1795x_{n-7}-128x_{n-8}+x_{n-9}
	,\qquad (n\ge10),
	\end{multline}
	the initial values are in Table \ref{table:typeof_vertices4d}.
	The sequences $\{a_n\}$, $\ldots$, $\{e_n\}$  can also be described by  
	\begin{equation}\label{eq:4d_recur_a-e}
	x_n=12x_{n-1}-37x_{n-2}+37x_{n-3}-12x_{n-4}+x_{n-5},\qquad (n\ge6),
	\end{equation} and sequences $\{a_n\}$, $\{b_n\}$ satisfy the equation 
	\begin{equation}
	x_n=4x_{n-1}-4x_{n-2}+x_{n-3},\qquad (n\ge4).
	\end{equation}
\end{theorem}

\begin{proof} According to \cite{NL_hyppyr}, as 
	\begin{equation}
	\vec{M}=\begin{pmatrix}
	1&1&0&0&0&0&0&0&0&\frac32\\ 
	1&2&0&0&0&0&0&0&0&0\\ 
	\frac23 &0&1& \frac23 &0&0&0&0&0&0\\ 
	0&1&\frac32&2&\frac52&0&0&0&0&0 \\
	0&0&3&4&6&0&0&0&0&0\\ 
	0&0&\frac14&0 &0&1&\frac12&0&0&0\\ 
	0&0&0&\frac13&0&2&2&\frac53&0&0 \\ 
	0&0&0&0&\frac12&6&6&6&6&0\\ 
	0&0&0 &0&0&94&97&101&107&0\\ 
	0&0&0&0&0&0&0&0&0&1\\ 
	\end{pmatrix}
	\end{equation}
	is the coefficients matrix of  \eqref{eq:4d_seq01v} with $rank(\vec{M})=10$, then the equation from minimal polynomial of matrix $\vec{M}$ equals to the characteristic equation of any sequence from \eqref{eq:4d_seq01v} and $s_n$. 
	The minimal polynomial of $\vec{M}$ is 
	\begin{eqnarray}\label{eq:4d_karpoly}
	p(x)&=&x^9-128x^8+1795x^7-8837x^6+19239x^5 \nonumber \\
	& & \quad \quad\quad \quad\quad\quad\quad\quad\quad\quad -19239x^4+8837x^3-1795x^2+128x-1 \nonumber \\
	&=&\left( x-1 \right)  \left( {x}^{2}-3\,x+1 \right)  
	\left( {x}^{2}-8\,x+1 \right)  
	\left( {x}^{4}-116\,{x}^{3}+366\,{x}^{2}-116\,x+1\right). 
	\end{eqnarray}
	Thus the recurrences can be describe by \eqref{eq:4d_recur01}. (The calculation was made by the help of software \textsc{Maple}.)  
	
	As $a_{n+1}$, $b_{n+1}$, $c_{n+1}$, $d_{n+1}$, $e_{n+1}$ and $v_{n+1}$ are independent from $f_n$, $g_n$, $h_n$  and $k_n$, they form a system of homogeneous recurrence equations again with matrix 
	$\vec{M}_{1}=\left( \begin{smallmatrix}
	1   & 1   & 0&0&0&\rfrac{3}{2}   \\ 
	1   & 2   & 0 & 0&0&0   \\ 
	\rfrac{2}{3}   & 0   & 1  &\rfrac{2}{3} & 0 &0  \\ 
	0  & 0   & 3 & 4&5&0   \\ 
	0  & 0   & 0 &0 & 0 &1  \\ 
	\end{smallmatrix}\right)$, where the minimal polynomial is 
	$x^5-12x^4+37x^3-37x^2+12x-1=(x-1)(x^2-8x+1)(x^2-3x+1)$. So  \eqref{eq:4d_recur_a-e} also holds. Similarly, we can reduce the degree of recursion  for  $a_{n}$ and $b_{n}$.
\end{proof}

\begin{remark}	In the Euclidean Pascal's simplex the equation system \eqref{eq:4d_seq01v} also holds with suitable initial values. In this case, there is no type vertices $B$, $D$, $E$, $G$, $H$, $K$, so $b_i=d_i=e_i=g_i=h_i=k_i=0$ for any $i$. Thus the hyperbolic Pascal simplex is not only the geometric but also the algebraic generalization of Pascal's simplex.   
\end{remark}

\begin{remark}
	The ratios of the numbers of the vertices from level to level tend to the largest eigenvalue of the matrix $\vec{M}$ (or largest solution of polynomial \eqref{eq:4d_karpoly}). Thus, the growing ratio of \hps is $\alpha_1\approx 112.763$, on the contrary it is $1$ in the Euclidean case.    
\end{remark}

\section{Sum of the values on levels in hyperbolic Pascal simplex}

The sum of the values of the elements on  level $n$ in the classical Pascal's simplex is $4^n$ (\cite{B}). In this section we determine it in case of the hyperbolic Pascal simplex. 

Denote respectively $\hat{a}_{n}$, $\hat{b}_{n}$, $\hat{c}_{n}$, $\hat{d}_{n}$,  $\hat{e}_{n}$, $\hat{f}_{n}$, $\hat{g}_{n}$, $\hat{h}_{n}$, $\hat{k}_{n}$  and $\hat{v}_{n}$, the sums of the values of  vertices type $A$, $B$, $C$, $D$, $E$, $F$, $G$, $H$, $K$  and $1$ on level $n$, and let $\hat{s}_{n}$ be the sum of all the values. 
From Figure~\ref{fig:gowing4d_edge}, \ref{fig:gowing4d_face} and \ref{fig:gowing4d_simlex} the results of Theorem \ref{th:4d_recursum} can be read directly. 
For example, all  vertices type $A$, $B$ and $1$ on level $i$ generate, respectively, two, two and three vertices type $A$ on level $i+1$ and it gives the first equation of \eqref{eq:4d_sum}.
Table \ref{table:typeof_vertices4d_sum} shows the sum of the values of the vertices on levels up to level 10.

\begin{theorem} \label{th:4d_recursum}
	If $n\geq1$, then
	\begin{equation}\label{eq:4d_sum}
	\begin{aligned}
	\hat{a}_{n+1}&= 2\hat{a}_n+2\hat{b}_n+3\hat{v}_n,\\
	\hat{b}_{n+1}&=  \hat{a}_n+2\hat{b}_n,\\
	\hat{c}_{n+1}&= 2\hat{a}_n+3\hat{c}_n+2\hat{d}_n,\\
	\hat{d}_{n+1}&= 2\hat{b}_n+3\hat{c}_n+4\hat{d}_n+5\hat{e}_n,\\
	\hat{e}_{n+1}&= 3\hat{c}_n+4\hat{d}_n+6\hat{e}_n,\\
	\hat{f}_{n+1}&= \hat{c}_n+4\hat{f}_n+2\hat{g}_n,\\
	\hat{g}_{n+1}&= \hat{d}_n+6\hat{f}_n+6\hat{g}_n+5\hat{h}_n,\\
	\hat{h}_{n+1}&= \hat{e}_n+12\hat{f}n+12\hat{g}_n+12\hat{h}_n+12\hat{k}_n,\\
	\hat{k}_{n+1}&=  94\hat{f}_n+97\hat{g}_n+101\hat{h}_n +107\hat{k}_n \\
	\hat{v}_{n+1}&=  \hat{v}_n,
	\end{aligned}
	\end{equation}
	with zero and $\hat{v}_1=4$ initial values.
\end{theorem}

\begin{table}[!htb]
	\centering \setlength{\tabcolsep}{0.3em}
	\begin{tabular}{|c||c|c|c|c|c|c|c|c|c|c|c|}
		\hline
		$n$   & 0 & 1 & 2  & 3  & 4  & 5    & 6     & 7        & 8          & 9            & 10             \\ \hline\hline
		$\hat{a}_n$ & 0 & 0 & 12 & 36 & 108& 348  & 1164  & 3948     & 13452      & 45900        & 156684         \\ \hline
		$\hat{b}_n$ & 0 & 0 & 0  & 12 & 60 & 228  & 804   & 2772     & 9492       & 32436        & 110772         \\ \hline
		$\hat{c}_n$ & 0 & 0 & 0  & 24 & 144& 840  & 5808  & 48552    & 458736     & 4588008      & 46916592       \\ \hline
		$\hat{d}_n$ & 0 & 0 & 0  & 0  & 96 & 1296 & 14400 & 152592   & 1592448    & 16530384     & 171272832      \\ \hline
		$\hat{e}_n$ & 0 & 0 & 0  & 0  & 72 & 1248 & 15192 & 166176   & 1753080    & 18264480     & 189472440      \\ \hline
		$\hat{f}_n$ & 0 & 0 & 0  & 0  & 24 & 240  & 2280  & 26880    & 667944     & 51411168     & 5797305000     \\ \hline
		$\hat{g}_n$ & 0 & 0 & 0  & 0  & 0  & 240  & 5976  & 255936   & 24140328   & 2793536160   & 331243298952   \\ \hline
		$\hat{h}_n$ & 0 & 0 & 0  & 0  & 0  & 360  & 38400 & 4458168  & 528618816  & 62831416920  & 7469847072960  \\ \hline
		$\hat{k}_n$ & 0 & 0 & 0  & 0  & 0  & 2256 & 323592& 39296736 & 4682378232 & 556809369792 & 66200381333976 \\ \hline
		$\hat{s}_n$ & 1 & 4 & 16 & 76 & 508& 7060 & 407620& 44411764 & 5239632532 & 622525195252 & 74007676940212 \\ \hline
	\end{tabular}
	\caption{Sum of values of vertices $n\leq10$}
	\label{table:typeof_vertices4d_sum}
\end{table}

\begin{theorem}\label{th:4d_sumvertex4q}  
	The sequences $\{\hat{a}_n\}$, $\ldots$, $\{\hat{k}_n\}$  and $\{\hat{s}_n\}$ can be described by the same tenth order linear homogeneous recurrence sequence 
	\begin{multline}\label{eq:4d_sumrecurcde}
	\hat{x}_n=147\hat{x}_{n-1}-3635\hat{x}_{n-2}+36277\hat{x}_{n-3}-175292\hat{x}_{n-4}+445156\hat{x}_{n-5}
	-608920\hat{x}_{n-6}\\
	+438532\hat{x}_{n-7}-151320\hat{x}_{n-8}+19344\hat{x}_{n-9}-288\hat{x}_{n-10}
	,\qquad (n\ge11),
	\end{multline}
	the initial values are in Table \ref{table:typeof_vertices4d_sum}.
	The sequences $\{\hat{a}_n\}$, $\ldots$, $\{\hat{e}_n\}$  can also be  described by  
	\begin{equation}
	\hat{x}_n = 18  \hat{x}_{n-1}-99\hat{x}_{n-2}+226\hat{x}_{n-3} -224\hat{x}_{n-4}+92\hat{x}_{n-5}-12\hat{x}_{n-6},\qquad (n\ge7),
	\end{equation}
	and sequences $\{\hat{a}_n\}$, $\{\hat{b}_n\}$  can  also be  formed by 
	\begin{equation}
	\hat{x}_n = 5\hat{x}_{n-1}-6\hat{x}_{n-2} +2\hat{x}_{n-3},\qquad (n\ge4).
	\end{equation}
\end{theorem}

The proof of this theorem is the same  step by step as the proof of Theorem~\ref{th:4d_numvertex4q}. 
The factorised minimal polynomial (now it is the same as the  characteristic polynomial) of the  coefficients matrix of  \eqref{eq:4d_sum} is 

\begin{eqnarray} \label{eq:4d_karpoly_sum}
p(x)&=&x^{10}-147x^9+3635x^8-36277x^7+175292x^6-445156x^5  \\
&& \quad\quad\quad\quad\quad\quad +608920x^4-438532x^3+151320x^2-19344x+288 \nonumber \\
&=&
\left( x-1 \right)  \left( {x}^{2}-4\,x+2 \right)  
\left( {x}^{3}-13 \,{x}^{2}+28\,x-6 \right) \nonumber\\
&& \quad\quad\quad\quad\quad\quad\quad\quad\quad\quad
\cdot \left( {x}^{4}-129\,{x}^{3}+1214\,{x}^{2}-1428\,x+24 \right).\nonumber
\end{eqnarray}

\begin{remark}
	The growing ratio of  values of \hps is $\approx\!118.89$, while it is $4$ in Euclidean case.    
\end{remark}

\end{document}